\documentclass{amsart}

\usepackage[latin1]{inputenc}
\usepackage{amsfonts}
\usepackage{amsmath}
\usepackage{amsthm}
\usepackage{amssymb}
\usepackage{latexsym}
\usepackage{enumerate}
\usepackage{bbm}
\usepackage{enumitem}
\usepackage{imakeidx}

\newtheorem{theorem}{Theorem}
\newtheorem{lemma}[theorem]{Lemma}
\newtheorem{proposition}[theorem]{Proposition}
\newtheorem{corollary}[theorem]{Corollary}

\newtheorem{question}[theorem]{Question}

\newtheorem{definition}[theorem]{Definition}

\numberwithin{equation}{section}

\DeclareMathOperator{\supp}{supp}

\DeclareMathOperator{\proj}{Proj}

\newcommand{\N}{\mathbb{N}}

\newcommand{\A}{\mathcal{A}}
\newcommand{\C}{\mathbb{C}}
\newcommand{\E}{\mathbb{E}}
\newcommand{\Q}{\mathbb{Q}}
\newcommand{\R}{\mathbb{R}}

\newcommand{\PP}{\mathbb{P}}

\newcommand{\forces}{\Vdash}
\newcommand{\dom}{\text{dom}}

\newcommand{\BB}{\mathcal{B}}
\newcommand{\I}{\mathcal{I}}

\newcommand{\QQ}{\mathcal{Q}(\ell_2)}

\newcommand{\leqK}{\leq^\mathcal{K}}
\newcommand{\downset}[1]{\langle #1 \rangle}

\begin{document}

\author{Damian G\l odkowski}
\address{Institute of Mathematics, 
University of Warsaw, ul. Banacha 2, 02-097 Warszawa, Poland}
\email{\texttt{d.glodkowski@uw.edu.pl}}

\author{Piotr Koszmider}
\address{Institute of Mathematics of the Polish Academy of Sciences,
ul.  \'Sniadeckich 8,  00-656 Warszawa, Poland}
\email{\texttt{piotr.math@proton.me}}

\thanks{The first-named author was partially supported by the NCN (National Science
Centre, Poland) research grant no.\ 2021/41/N/ST1/03682.}

\thanks{The second-named author was partially supported by the NCN (National Science
Centre, Poland) research grant no.\ 2020/37/B/ST1/02613.}

\thanks{For the purpose of Open Access, the authors have applied a CC-BY public
 copyright license to any Author Accepted Manuscript (AAM) version arising from this submission.}

\subjclass[2020]{}

\title[Products that do not embed into $\QQ$]{Products of C*-algebras that do not embed into the Calkin algebra}

\begin{abstract} We consider the Calkin algebra $\QQ$, i.e.,  the quotient of the algebra $\mathcal B(\ell_2)$ of all
bounded linear operators on the separable Hilbert space $\ell_2$ divided by the ideal $\mathcal K(\ell_2)$
of all compact operators on $\ell_2$. 
We show that in the Cohen model of set theory {\sf ZFC} there is no embedding of the product 
$(c_0(2^\omega))^\N$ of infinitely many copies of the abelian C*-algebra $c_0(2^\omega)$ into $\QQ$ (while
$c_0(2^\omega)$ always embeds into $\QQ$). 
This enlarges the collection of the known  examples due to Vaccaro and to McKenney and Vignati of abelian algebras, 
asymptotic sequence algebras, reduced products and coronas of stabilizations which
consistently do not
embed into the Calkin algebra. As in the Cohen model the rigidity of quotient structures fails in general,
our methods do not rely on these rigidity phenomena as is the case of most
examples mentioned above. 
The results should be
considered in the context of the result of Farah, Hirshberg and Vignati which says that
consistently all C*-algebras of density up to $2^\omega$ do embed into $\QQ$.
In particular, the algebra $(c_0(2^\omega))^\N$  consistently embeds into the Calkin algebra as well.
\end{abstract}

\maketitle

\section{Introduction}

For a separable Hilbert space $H$ we denote by $\mathcal{B}(H)$ the C*-algebra
 of all bounded linear operators on $H$ and by $\mathcal{K}(H)$ the ideal 
 of $\mathcal{B}(H)$ consisting of compact operators. 
 The quotient $\mathcal{Q}(\ell_2)=\mathcal{B}(\ell_2)/\mathcal{K}(\ell_2)$ is 
 called the Calkin algebra. The significance of the Calkin algebra for
 the understanding of the algebra $\mathcal{B}(\ell_2)$ has been realized
 since its inception in \cite{calkin}. This paper concerns the investigation 
 of the class of all
 C*-algebras which can be embedded\footnote{Unless stated otherwise, 
 by an embedding we mean an injective $*$-homomorphism of C*-algebras.} 
 into $\mathcal{Q}(\ell_2)$ (necessarily of density up to $2^\omega$
 as the density of $\QQ$ is $2^\omega$).  This class is denoted $\E$ following \cite{vaccaro-ijm}.
 
 It follows from
 the result of \cite{shelah-usvyatsov} which says that it is consistent that
 no Banach space of density $2^\omega$ contains isometric copies
 of all Banach spaces of density $2^\omega$ (``isometric''  here may even be replaced by ``isomorphic'' 
 by \cite{pk-brech-univ}), that it is consistent that
 some abelian C*-algebra of density $2^\omega$ is not in $\E$ (as every Banach space isometrically
 embeds into the C*-algebra of continuous functions on its dual ball with the weak$^*$ topology).
 However, the witnesses of the nonuniversality in \cite{shelah-usvyatsov}  or \cite{pk-brech-univ} are 
 fundamentally nonconstructive.
 
 Concrete consistent examples of (abelian) C*-algebras of density $2^\omega$ not in $\E$
 have been produced at least since the doctoral thesis of Vignati \cite{vignati-thesis} 
 which was followed by other algebras or other consistent hypotheses in 
 \cite{vignati-mckenney}, \cite{vaccaro-thesis}, \cite{vaccaro-ijm}.
 
 On the other hand, on the positive side, it is not difficult to see that $\mathcal{B}(\ell_2)$
  (obviously with many of its subalgebras) is in $\E$, and so all separable C*-algebras 
  belong to the class $\E$ (see e.g. \cite[Corollary II.6.4.10]{blackadar}).
 Perhaps the simplest example of a C*-algebra that does not embed into $\mathcal{B}(\ell_2)$ but
 does embed into $\QQ$ is 
 the abelian algebra $c_0(2^\omega)$.  The existence of such an embedding follows from the existence
 of an uncountable almost disjoint (i.e., with finite pairwise intersections) family
 of infinite subsets of $\N$.
 
 Moreover, it was shown by Farah, Hirshberg and Vignati in \cite{farah-universal} that all C*-algebras of
 density equal to the first uncountable cardinal $\omega_1$ are in $\E$.
  In particular, under the continuum hypothesis {\sf CH}, the Calkin algebra
 is universal in the class of C*-algebras of density $2^\omega$. Also
 Martin's axiom implies that all C*-algebras of density strictly less than $2^\omega$
 belong to $\E$ (\cite{farah-ma}). 
 
 The main purpose of this note is to provide concrete
 examples  of C*-algebras that consistently do not embed into the Calkin algebra, i.e., are not in $\E$ which are not
 covered by the results of \cite{vignati-thesis}, \cite{vignati-mckenney}, \cite{vaccaro-thesis}, \cite{vaccaro-ijm},
 are obtained by rather different methods, 
 and exist in the absence of the combinatorial or forcing principles
 considered in \cite{vignati-thesis} \cite{vignati-mckenney},  \cite{vaccaro-ijm} and more
 interestingly, in the absence of the general rigidity of quotient structures:
 
 \begin{theorem}\label{main} In the Cohen model\footnote{By the Cohen model we mean a model 
 obtained from a model of {\sf CH} by adding $\omega_2$ Cohen reals.}
 the abelian C*-algebra $(c_0(2^\omega))^\N$ does not embed into the Calkin algebra
 $\mathcal{Q}(\ell_2)$.
 As a consequence the following types of C*-algebras do not embed into the Calkin algebra:
 \begin{itemize}
 \item products $\prod_{n\in \N}\A_n$,
 \item reduced product algebras  $$\prod_{n\in \N}{\A_n}/\bigoplus_{\mathcal I}{\A_n},$$
 \item the corona algebras ${\mathcal Q}({\A}\otimes {\mathcal K})$ of the stabilization
 ${\A}\otimes {\mathcal K}$ of  $\A$,
 \end{itemize}
  where each of the C*-algebras ${\A}, {\A}_n$ admits a pairwise orthogonal family of 
 projections of cardinality $2^\omega$ and $\I$ is any ideal of subsets of $\N$  such that
 the Boolean algebra $\wp(\N)/\mathcal I$ is infinite.
 \end{theorem}

 \begin{proof}
     See Theorem \ref{reduced_products}, Corollary \ref{cor_products} and Corollary \ref{corona-no-embedding}. 
 \end{proof}
 
 Reduced products of C*-algebras have been considered at  least in
 \cite{vignati-thesis}, \cite{vignati-mckenney}, \cite{farah-book}, \cite{farah-shelah}, \cite{ghasemi-ijm},
 \cite{ghasemi-jsl}, while the coronas of stabilizations were considered in the context of embeddings into
 the Calkin algebra in 
 \cite{vignati-mckenney}, \cite{vaccaro-ijm}, \cite{farah-kazhdan}. It is also interesting to compare 
 the Calkin algebra with other corona algebras:
 Farah shows in \cite{farah-kazhdan} that $\QQ$ is not isomorphic to the corona of
 the stabilization of the Cuntz algebra i.e. 
 $\mathcal{Q}(\mathcal{O}_\infty\otimes \mathcal{K}(\ell_2))$, 
 though these algebras are not distinguishable from the K-theoretical point of view. 
 This research is followed in \cite{farah-szabo}.
 As mentioned in \cite{farah-kazhdan} the problem of the existence of such isomorphisms is important, 
 since they may induce K-theory reversing automorphisms of $\QQ$ 
 (such automorphisms are not known to exist consistently). 
 
 Let us compare
 our Theorem \ref{main} to some of the results of the papers mentioned above.
 Concerning the products and direct sums one may make some conclusion from the results in \cite{vaccaro-ijm}, where
 it is proved that the tensor products
 $\mathcal{Q}(\ell_2)\otimes \mathcal A$ do not embed into $\mathcal{Q}(\ell_2)$ for any unital
 infinite dimensional $\A$ (see the proof rather than the statement of Theorem 1.2). Consider
 the C*-algebra  $\A=C(\omega+1)$ of continuous functions on a convergent sequence with its limit.
 The tensor product $\mathcal{Q}(\ell_2)\otimes C(\omega+1)$ is
  isomorphic to $C(\omega+1, \mathcal{Q}(\ell_2))$ (\cite{williams})
 and does not embed into the Calkin algebra by \cite{vaccaro-ijm}.
 This algebra $\A$ seems very close to 
 $c_0(\mathcal{Q}(\ell_2))=\bigoplus_{Fin} \mathcal{Q}(\ell_2)$ which clearly embeds into $\QQ$,
 also $\A$  embeds into $(\mathcal{Q}(\ell_2))^\N$. So 
 the result of \cite{vaccaro-ijm} implies that $(\mathcal{Q}(\ell_2))^\N$ does not embed into the Calkin
 algebra.  However our result allows much smaller C*-algebras $\A_n$ in place of $\QQ$ at
 the price of taking full products.
 
Concerning reduced products associated with ideals  Corollary 5.3.14   \cite{vignati-thesis}  
or Theorem 9.1. of \cite{vignati-mckenney}
 on reduced products may be used for one-dimensional C*-algebras to conclude that the abelian C*-algebras 
 of continuous functions $C(K)$
 where $K$ is the Stone space of a Boolean algebra $\wp(\N)/{\mathcal I}$  for
 a meager, dense ideal in $\wp(\N)$ are not in $\E$ assuming
 the Proper Forcing Axiom {\sf PFA} or the Open Coloring Axiom {\sf OCA} together with
 Martin's axiom {\sf MA}. Here we note that by Talagrand's characterization of meager ideals 
 (\cite{talagrand}, Theorem  4.1.2 of \cite{bartoszynski}) such algebras $C(K)$ contain
big Boolean algebras of projections $p$ such that $pC(K)$ is again induced by meager dense ideal and so
 $pC(K)$ is not in $\E$ as well in the above model. However,  $pc_0(2^\omega)^\N$ is a product of finite dimensional
 algebras for any projection $p$ in $c_0(2^\omega)^\N$ and so it is isomorphic to $\ell_\infty$ and hence is in $\E$.
    Also the reduced products of our Theorem \ref{main}
 are associated with any ideal $\I$ such that $\wp(\N)/\I$ is infinite but we  have requirements on the algebras $\A_n$  while
 the ideals in  \cite{vignati-thesis}   and \cite{vignati-mckenney} need to be meager (and dense and contain
 the finite sets) with the algebras  $\A_n$s required only to be nontrivial.

 Another concrete consistent example of an abelian C*-algebra not in $\E$ 
 is $C_0(\omega_2)$ in the Cohen model, where $\omega_2$ is considered with the order
 topology (\cite{vaccaro-thesis}). 
 Note that this algebra does not embed into $(c_0(2^\omega))^\N$ as
 a  strictly increasing well-ordered sequence of length $\omega_2$
   of projections in $(c_0(2^\omega))^\N$ would induce a  strictly increasing sequence of projections
   of the same length in $c_0(2^\omega)$ which is impossible.

Let us finish this discussion of our result with a comparison of the Cohen model, where
we work with the models of \cite{vignati-thesis}, \cite{vignati-mckenney}, \cite{vaccaro-ijm}
where most  of the discussed examples outside of $\E$ can be found. 
In the above papers the authors work under powerful and  elegant axioms (forcing axioms
and/or the Open Coloring Axiom {\sf OCA}) which imply the rigidity of quotient structures
like $\wp(\N)/Fin$, $\ell_\infty/c_0$ or $\QQ$ in the sense that all automorphisms (or even endomorphism
in case of $\QQ$ and \cite{vaccaro-ijm}) of these
structures are trivial. In the Cohen model the rigidity fails in general, for example
there are automorphisms of $\wp(\N)/Fin$ which are not induced by a bijection between some 
cofinite subsets of $\N$ as shown in \cite{shelah-steprans} (it seems still unknown if 
outer automorphisms of $\QQ$ exist in this model).  So our results
shed some light on the relation between the rigidity phenomena and the class $\E$.
As in the case of the algebra $\wp(\N)/Fin$ it may be also worthy
to develop the theory of the Calkin algebra in the Cohen model as
``the Cohen model is probably the most intensively investigated model of
the negation of {\sf CH} of all'' (\cite{dow-topappl} \S7).

The structure of the paper is as follows. The next section contains preliminaries. 
These include notation, terminology and lemmas concerning basic properties 
of discussed objects. In Section 3 we prove the key fact on the impossibility of 
embedding the abelian C*-algebra $c_0(2^\omega)^\N$
 into the Calkin algebra in the Cohen model. In Section 4 we conclude the components of Theorem \ref{main}.
The last section includes some concluding remarks and questions.

\section{Preliminaries}

\subsection{Notation} Most of the basic symbols used in the paper should be standard. The symbol $\N$ denotes the set of non-negative integers, $\Q$ is the set of rational numbers, $\omega_n$ stands for the $n$-th uncountable cardinal number, $2^\omega$ is the cardinality of the set of real numbers. The symbol $f\upharpoonright A$ denotes the restriction of a function $f$ to a set $A$. The family of subsets of a set $A$ of cardinality $\lambda$ is denoted by $[A]^\lambda$, and the family of finite subsets of $A$ is denoted by $[A]^{<\omega}$. The symbol $<_{lex}$ stands for the lexicographic order on pairs of ordinals numbers. The identity function on a set $A$ is denoted by $Id_A$. More specific terminology will be explained in next subsections of this section. For unexplained terminology see \cite{farah-book} and \cite{jech}. 

\subsection{Products, reduced products and coronas}

Following notation from \cite[Chapter 2]{farah-book}, given C*-algebras $\A_n$ for $n\in\N$
 we define the product $\prod_{n\in\N} \A_n$ as the algebra of bounded 
 sequences $(a_n)_{n\in \N}$ such that $a_n\in \A_n$ for each $n\in\N$ 
 with the standard supremum norm and coordinate-wise operations. 
 The product of countably many copies of a C*-algebra $\A$ is denoted by $\A^\N$.

 Given an ideal $\I$ of subsets of $\N$ we denote by $\bigoplus_{\I} \A_n$ 
 the closure of $\{(a_n)_{n\in\N}\in \prod_{n\in\N}\A_n: \{n\in \N: a_n\neq 0\}\in \I\}$. 
 If $\I=Fin$ is the ideal of finite subsets of $\N$, then $\bigoplus_{Fin}\A_n$ is the standard direct sum of $\A_n$'s.

Let us recall some basic facts concerning multiplier algebras and coronas. 

\begin{definition}
    A C*-algebra $\mathcal{M}(\mathcal{A})\supseteq\mathcal{A}$ 
    is called the \textbf{multiplier algebra} of a C*-algebra 
    $\mathcal{A}$, if $\mathcal{A}$ is an essential ideal in 
    $\mathcal{M}(\mathcal{A})$ (i.e. $\mathcal{A}^\bot_{\mathcal{M}(\mathcal{A})}= \{0\}$,
     where $\mathcal{A}^\bot_{\mathcal{D}}=\{x\in \mathcal{D}: \mathcal{A}x=\{0\}\}$) 
     and for every C*-algebra $\mathcal{D}$ containing $\mathcal{A}$ 
     as an ideal the identity map $Id\colon \mathcal{A}\to\mathcal{B}(\mathcal{A})$ 
     has a unique extension to $\mathcal{D}$ with kernel $\mathcal{A}^\bot_\mathcal{D}$. 

    The quotient algebra $\mathcal{Q}(\mathcal{A})= \mathcal{M}(\mathcal{A})/\mathcal{A}$
     is called the \textbf{corona} of $\mathcal{A}$.
\end{definition}

It is well-known that $\mathcal{M}(\mathcal{K}(\ell_2))\equiv \mathcal{B}(\ell_2)$ 
and $\mathcal{Q}(\mathcal{K}(\ell_2))\equiv \QQ$. 
If $X$ is a locally compact Hausdorff space, then 
$\mathcal{M}(C_0(X))\equiv C(\beta X)$ and 
$\mathcal{Q}(C_0(X))\equiv C(\beta X\backslash X)$, 
and so the multiplier algebra should be seen as 
the non-commutative analogue of the Stone-\v{C}ech 
compactification of a topological space, and corona as 
the non-commutative analogue of the Stone-\v{C}ech remainder. 

The multiplier algebra of $\A$ may be characterized as the algebra of double
 centralizers of $\A$ (see \cite[Theorem II.7.3.4]{blackadar}). 

\begin{definition}
    A pair $(L, R)$ of linear maps on a C*-algebra $\mathcal{A}$
     is a \textbf{double centralizer} of $\mathcal{A}$, if $xL(y)=R(x)y$ for all $x,y\in \mathcal{A}$. 
\end{definition}

For more information on multiplier algebras and coronas see \cite{blackadar} or \cite{farah-book}. 

For a set $A$ by $c_0(A)$ we mean the algebra of (complex) sequences on $A$
 converging to 0 (i.e. $(a_\alpha)_{\alpha\in A}\in c_0(A)$, if  for every $\varepsilon>0$ 
 the set $\{\alpha\in A: |a_\alpha|>\varepsilon\}$ is finite) with pointwise multiplication and the supremum norm. 
 We put $(a_\alpha)_{\alpha\in A}^*= (\overline{a}_\alpha)_{\alpha\in A}$.

The following lemma follows from the existence of an almost disjoint family 
of subsets of $\N$ of cardinality $2^\omega$ (cf. the proof of the implication $(b) \to (c)$ in \cite{pk-brech-sums}).

\begin{lemma}\label{c_0(c)}
    There is an embedding of $c_0(2^\omega)$ into $\ell_\infty/c_0$.
\end{lemma}

\subsection{Operators on separable Hilbert spaces}

Throughout the paper by $\ell_2$ we mean the separable complex Hilbert space consisting of 
square-summable sequences $(a_n)_{n\in\N}$ of complex numbers with the standard
 inner product $\downset{(a_n)_{n\in\N}, (b_n)_{n\in\N}} = \sum_{n=1}^\infty a_n\overline{b}_n$. 


An element $p$ of a C*-algebra $\mathcal{A}$ is a projection, if $p=p^2=p^*$. 
If $P\in \mathcal{B}(\ell_2)$, then $P$ is a projection if and only if it is an orthogonal 
projection onto a subspace of $\ell_2$. On the set of all projections $\proj(\mathcal{A})$ 
on $\mathcal{A}$ we introduce the ordering given by $p\leq q$ if and only if $qp=p$. 
For $P,R\in \proj(\mathcal{B}(\ell_2))$ we denote $P\leqK Q$, if $QP-P\in \mathcal{K}(\ell_2)$. 
Note that $P\leqK Q$ if and only if $\pi(P)\leq \pi(Q)$, where
 $\pi\colon \mathcal{B}(\ell_2)\to \QQ$ is the canonical quotient map. 
 
 \begin{definition}\label{def-fin} $Fin_1(\sqrt\Q)$ denotes the set of  all elements $v$ of $\ell_2$
 such that 
 \begin{enumerate}
 \item $\|v\|=1$
 \item $|v(n)|^2\in \Q$ for each $n\in \N$ and 
 \item $\{n\in \N: v(n)\not=0\}$ is finite.
 \end{enumerate}
 \end{definition}
 
 \begin{lemma}\label{dense} Let $n\in \N$ and 
$X_n=\{v\in \ell_2: \|v\|=1\ \&\  v(m)=0 \ \hbox{\rm for all}\  m<n\}$. 
Then $Fin_1(\sqrt\Q)\cap X_n$ is dense in $X_n$.
\end{lemma}

\begin{definition}\label{def-very}
    A sequence $(v_n)_{n\in \N}\subseteq  Fin_1(\sqrt\Q)$ is called \textbf{very orthonormal} 
    if $\downset{v_n, v_m}=0$ for every distinct $n, m\in \N$ and
    $v_n(m)=0$ for all $m<n$.
\end{definition}

The following lemma follows from the fact that compact operators on $\ell_2$ may be approximated by finite-dimensional operators.

\begin{lemma}\label{compact_Conway}
    Let $K\in \mathcal{K}(\ell_2)$. Let $(v_n)_{n\in \N}$ be an orthonormal sequence in $\ell_2$. Then 
    $$\lim_{n\to \infty} \|K(v_n)\|=0.$$
\end{lemma}

\begin{lemma}\label{compact-very} Suppose that $K\in \BB(\ell_2)$ is a compact operator. 
Then there is  a function $f_K:\N\rightarrow \N$ such that  for every
 very orthonormal sequence $(v_n)_{n\in \N}$,  for
 every $n\in \N$ and every $m\geq f_K(n)$ we have 
$$\|K(v_m)\|\leq {\frac{1}{n+1}}.$$
\end{lemma}
\begin{proof} Let $Q_n\in \BB(\ell_2)$ be the projection onto
$\{v\in \ell_2: v(m)=0 \ \hbox{\rm for all}\ m< n\}$. Since $K$ may be approximated by
 finite-dimensional operators we have $$\lim_{n\to \infty}\|KQ_n\|=0.$$ Picking
$f_K(n)\in \N$ such that $\|KQ_m\|\leq 1/(n+1)$ for all $m\geq f_K(n)$
we guarantee that $\|K(v_m)\|\leq 1/(n+1)$ as $Q_m(v_m)=v_m$.
\end{proof}

\begin{lemma}\label{projection-very} Suppose that $E\in \BB(\ell_2)$ is an infinite dimensional projection.
Then there is a very orthonormal sequence $(v_n)_{n\in \N}$ such that
$$\|E(v_n)\|\geq 1-{\frac{1}{n+1}}$$
 for  every $n\in \N$.
\end{lemma}
\begin{proof} Conctruct $(v_n)_{n\in \N}$ by recursion on $n\in \N$. Suppose that
we are given $v_0, \dots, v_n$. Let $P_k\in \BB(\ell_2)$ be the projection onto
$\{v\in \ell_2: v(m)=0 \ \hbox{\rm for all}\ m\geq k\}$. Using Definition \ref{def-very} find $k\in \N$
such that $P_k(v_m)=v_m$ for all $m\leq n$. 

Let $(e_n)_{n\in \N}$ be an orthonormal sequence included in the range
of $E$. The operator $EP_k$ is compact, so by Lemma \ref{compact_Conway}  there is $m\in \N$ such that 
$$\|EP_k(e_m)\|\leq 1/3(n+2).$$
As $E=EP_k+E(I-P_k)$, we have 
$$\|E(I-P_k)(e_m)\|\geq 1-1/3(n+2)\leqno(*)$$ 
and so in particular 
$$\|(I-P_k)(e_m)\|\geq 1-1/3(n+2).\leqno (**)$$
Let 
$$v={\frac{(I-P_k)(e_m)}{\|(I-P_k)(e_m)\|}}.$$
As $v$ belongs to the range of $(I-P_k)$, use Lemma \ref{dense}
to find $v_{n+1}\in Fin_1(\sqrt\Q)$ belonging to the  range  $(I-P_k)$ such 
that $\|v_{n+1}-v\|\leq 1/3(n+2)$. By $(**)$ we obtain
$$\|v_{n+1}-(I-P_k)(e_m)\|\leq \|v_{n+1}-v\|+\|v-(I-P_k)(e_m)\|\leq$$
$$\leq 1/3(n+2)+\Big|{\frac{1}{\|(I-P_k)(e_m)\|}}-1\Big|\|(I-P_k)(e_m)\|\leq 2/3(n+2),$$
and so
$$\|E(v_{n+1}-(I-P_k)(e_m))\|\leq 2/3(n+2)\leqno(***)$$
as $\|E\|=1$.
Finally $(*)$ and $(***)$ give that 
$\|E(v_{n+2})\|\geq 1-1/(n+2)$ as required.
\end{proof}

\begin{lemma}\label{rachunki} Suppose that $P, Q\in \BB(\ell_2)$ are projections and
$v\in \ell_2$ satisfies $\|v\|=1$, $\|P(v)\|, \|Q(v)\|> 99/100$. 
Then 
\begin{enumerate}
\item $\|P(v)-v\|, \|Q(v)-v\|<0.15$, 
\item $|\langle P(v), Q(v)\rangle|>0.5$
\end{enumerate}
\end{lemma}
\begin{proof}
By the Pythagorean Theorem we have $\|P(v)-v\|^2+\|P(v)\|^2=1$, so
$\|P(v)-v\|^2<199/10000$, so $\|P(v)-v\|<0.15$ which gives (1).
For (2) consider
$$\langle P(v), Q(v)\rangle=\langle v, v\rangle+
 \langle P(v)-v, Q(v)\rangle+ \langle P(v), Q(v)-v\rangle- \langle P(v)-v, Q(v)-v\rangle$$
and use (1).
\end{proof}

\subsection{Sets}
We will need some basic facts concerning $\Delta$-systems. 

\begin{definition}\label{delta_system_def}
    We say that a family of sets $\mathcal{A}$ is a $\Delta$\textbf{-system} 
    with \textbf{root} $\Delta$, if for every $A,B\in \mathcal{A}$ we have $A\cap B=\Delta$ whenever $A\neq B$. 
\end{definition}

\begin{lemma}\label{Delta_lemma_CH}
\cite[Theorem 9.19]{jech} Assume {\sf CH}. Then for every family 
of countable sets $\mathcal{A}$ of cardinality $\omega_2$ 
there is a $\Delta$-system $\mathcal{B}\subseteq \mathcal{A}$ with $|\mathcal{B}|=\omega_2$.
\end{lemma}

\begin{lemma}\label{RcapS}
Assume {\sf CH}\footnote{In  fact, we do not need to assume {\sf CH} here because we may use a weak version
of the $\Delta$-system like in \cite{witek} Lemma  4.12 which is true without any additional hypothesis.
  Since anyway we will need {\sf CH} for other purposes, we present the standard argument assuming {\sf CH}.}. 
  Suppose $(S_\alpha)_{\alpha<\omega_2}$ is
 a sequence of pairwise disjoint countable subsets of $\omega_2$ 
 and $(R_\alpha)_{\alpha<\omega_2}$ is a sequence of countable 
 subsets of $\omega_2$. Then there are $\xi, \eta<\omega_2$ such that $\xi\neq \eta$ and 
$$S_\xi\cap R_\eta = S_\eta\cap R_\xi = \varnothing.$$
\end{lemma}

\begin{proof}
    By Lemma \ref{Delta_lemma_CH} there is $A\in[\omega_2]^{\omega_2}$ 
    such that $(R_\xi)_{\xi\in A}$ is a $\Delta$-system with root $\Delta$. 
    The set $\Delta$ is countable, so there is $B\in[A]^{\omega_2}$ 
    such that $S_\xi\cap \Delta=\varnothing$ for $\xi\in B$. 
    Pick any $\xi \in B$. Since $R_\xi$ is countable, 
    there is $C\in[B]^{\omega_2}$ such that 
    $S_\eta\cap R_\xi = \varnothing$ for all $\eta\in C$. 
    Since $S_\xi$ is countable and the sets $(R_\eta\backslash \Delta)_{\eta \in C}$ 
    are pairwise disjoint, we can pick 
    $\eta\in C\backslash\{\xi\}$ such that $S_\xi \cap (R_\eta \backslash \Delta) = \varnothing$.
     It follows that $S_\xi\cap R_\eta = \varnothing$.
\end{proof}

\begin{lemma}\label{functions}
Suppose $A_n \in [\omega_2]^{\omega_2}, S_{n,\alpha} \in [\omega_2]^\omega$
 for $n\in\N, \alpha<\omega_2$ are such that for every $n\in \N$ 
 the family $(S_{n,\alpha})_{\alpha\in A_n}$ is a $\Delta$-system with 
 root $\Delta_n$. Assume that for each $\alpha \in A_n$ we have 
 $\Delta\cap S_{n,\alpha}=\Delta_n$, where $\Delta= \bigcup_{n\in \N} \Delta_n$. 
 Then for every $\xi < \omega_2$ there is $\gamma_\xi \in \omega_2^\N$ such that 
\begin{enumerate}[label=(\alph*)]
    \item $\gamma_\xi(n)\in A_n$ for $n\in \N$,
    \item for distinct $(\xi,n),(\eta,m)\in\omega_2\times\N$  we have
    $$(S_{n,\gamma_\xi(n)}\backslash \Delta_n)\cap (S_{m,\gamma_\eta(m)}\backslash \Delta_m)=\varnothing,$$
    \item $\gamma_\xi\cap \gamma_\eta = \varnothing$ for $\xi, \eta<\omega_2, \xi\neq \eta$.
\end{enumerate}
\end{lemma}

\begin{proof}
We construct $(\gamma_\xi)_{\xi<\omega_2}$ by induction on $\xi<\omega_2$
 and $n\in\N$. Fix $\xi<\omega_2$ and $n\in \N$. Suppose that
 we have constructed $\gamma_\eta$ for $\eta<\xi$ and $\gamma_\xi(m)$ for $m<n$. 
 Put $\delta=\sup\limits_{\eta<\xi, n\in\N} \gamma_{\eta}(n)$ and observe that $|A_n\backslash (\delta+1)|=\omega_2$.
  The set 
\begin{gather*}
    B=\bigcup_{(\eta,m) <_{lex} (\xi,n)} S_{m,\gamma_\eta(m)}\backslash\Delta
\end{gather*}
has cardinality at most $\omega_1$ and the family $(S_{n,\alpha}\backslash\Delta)_{\alpha<\omega_2}$ 
consists of non-empty pairwise disjoint sets, 
so there is $\beta\in A_n\backslash (\delta+1)$
 such that $S_{n,\beta}\cap B = \varnothing$. 
 We put $\gamma_\xi(n)=\beta$. 
 It follows directly from the construction that the conditions (a)-(c) are satisfied. 
\end{proof}

\subsection{The Cohen model}
By  $\PP$ 
we will denote the Cohen forcing adding $\omega_2$ reals i.e. 
$$\PP=\{p\colon \dom(p)\to \{0,1\} : \dom(p)\in [\omega_2]^{<\omega} \}$$
with the ordering given by $q\leq p$ if and only if $q\supseteq p$. 
The generic extension of a model of {\sf CH} obtained by
forcing with  $\PP$ is called the Cohen model.

\begin{definition}
    Let $\sigma\colon \omega_2 \to \omega_2$ be a permutation. 
    We define the \textbf{lift} $\sigma^\PP$ of $\sigma$ as the automorphism
     $\sigma^\PP\colon\PP\to \PP$ given by $\sigma^\PP(p)(\sigma(\alpha))= p(\alpha)$ for $p\in \PP, \alpha\in \dom(p)$. 
     If $\dot{x}=\{(\dot{y}_i, p_i): i\in I\}$ is a 
    $\PP$-name, then we define
    $$\sigma^{\dot\PP}(\dot{x})= \{(\sigma^{\dot\PP}(\dot{y}_i), \sigma(p_i)): i\in I\}$$
     (cf. \cite[p. 221]{jech}). 
\end{definition}

Note that if $\check{x}$ is the canonical name for and object $x$ of the ground model, 
then for any permutation $\sigma\colon \omega_2\to \omega_2$ we have $\sigma^{\dot\PP}(\check{x})=\check{x}$.

\begin{definition}
    Let $\sigma\colon \omega_2\to \omega_2$ be a permutation such 
    that $\sigma=\sigma^{-1}$, $\sigma[S_1]=S_2$, where $S_1, S_2 \subseteq \omega_2$. 
    We say that $p\in \PP$ is $(\sigma, S_1, S_2)$\textbf{-symmetric}, 
    if $\sigma^\PP(p\upharpoonright S_1)= p\upharpoonright S_2$. 
\end{definition}

\begin{lemma}\label{symmetric-symmetric}
 Let $\sigma\colon \omega_2\to \omega_2$ be a permutation such 
    that $\sigma=\sigma^{-1}$, $\sigma[S_1]=S_2$, where $S_1, S_2 \subseteq \omega_2$. 
    $p\in \PP$ is $(\sigma, S_1, S_2)$-symmetric if and only if 
    $p$ is $(\sigma, S_2, S_1)$-symmetric
\end{lemma}
\begin{proof} Note that the condition $\sigma=\sigma^{-1}$ implies
that $(\sigma^\PP)^2$ is the identity on $\PP$.
\end{proof}

\begin{lemma}\label{symmetric-condition}
    Let $S_1, S_2\subseteq \omega_2$ and suppose that $\sigma\colon \omega_2\to \omega_2$
     is a permutation such that $\sigma=\sigma^{-1}$,  $\sigma[S_1]=S_2$ and $\sigma\upharpoonright 
     S_1\cap S_2=Id_{S_1\cap S_2}$. 
     Suppose $p\in \PP$ is $(\sigma, S_1, S_2)$-symmetric and $q\leq p$ is such
      that $(\dom(q)\backslash \dom(p))\cap S_2 \subseteq S_1$. Let $$r=q\cup \sigma(q\upharpoonright S_1).$$
    Then $r\in\PP$ is  $(\sigma, S_1, S_2)$-symmetric. 
\end{lemma}

\begin{proof}
First, note that $r\in \PP$ since if 
$$\alpha\in \dom(q)\cap \dom(\sigma(q\upharpoonright S_1))=
\big((\dom(q)\backslash \dom(p))\cap S_2\big)\cup dom(p)\cap S_2,$$
 then either $\alpha\in S_1\cap S_2$ (and $r(\alpha)$
  is well-defined by the hypothesis that $\sigma\upharpoonright S_1\cap S_2=Id_{S_1\cap S_2}$) 
  or $\alpha\in \dom(p)\cap S_2$, and so $r(\alpha)$ is well-defined by the symmetry of $p$.

 Now we will show that $r$ is $(\sigma, S_1, S_2)$-symmetric. 
 It is clear from the definition and from the hypothesis that
 $\sigma[S_1]=S_2$ that $\supp(\sigma(r\upharpoonright S_1))=\supp(r \upharpoonright S_2)$.
  Fix $\alpha\in \dom(r)\cap S_1$. If $\alpha\in\dom(p)$, then $\sigma(r)(\alpha)=
   r(\sigma(\alpha))$ by the symmetry of $p$. If $\alpha\in \dom(q)\backslash \dom(p)$, 
   then we have $\sigma(r)(\alpha)= r(\sigma(\alpha))$ from the equality defining $r$.
\end{proof}

Let us recall the definition of a nice name.
\begin{definition}
A $\PP$-name $\dot{X}$ is a \textbf{nice name} for a subset of $M\in V$, 
if it is of the form $\dot{X}=\bigcup_{m\in M} \{\check{m}\}\times A_m$,
 where $A_m$ is an antichain in $\PP$ for each $m\in M$. 
 The set $\bigcup_{m\in M} \bigcup_{p\in A_m} \dom(p)$ is called 
 the \textbf{support} of $\dot{X}$ and is denoted by $\supp(\dot{X})$. 
\end{definition}

Note that since $\PP$ satisfies c.c.c. every nice name for a subset of $\N$ has a countable support.

\begin{lemma}\label{small-domain}
    Suppose $\dot x$ is a nice $\PP$-name with $\supp(\dot x)=S$ and $p\in\PP$
     forces that $\dot x$ is an element of the ground model. 
     Then there is $q\leq p$ with $\dom(q)\backslash \dom(p)\subseteq S$ that decides $\dot x$. 
\end{lemma}

\begin{lemma}\label{remarkkkkkkk}
    If $\dot{h}$ is a nice name such that $$\PP\forces \dot{h}\colon \N\to \N \ \text{is a function},$$
     then for every $p\in\PP$ and $n\in\N$ there is $q\leq p$ and $m\in \N$ such that 
    $$q\forces \dot{h}(\check n)=\check m$$ and $\dom(q)\backslash \dom(p)\subseteq \supp(\dot{h})$.
\end{lemma}

\begin{lemma}\label{permutation}\cite[Lemma 14.37]{jech} Suppose that $\phi(x_1, \dots,x_n)$ is a formula
 of the language of {\sf ZFC} with  $n\in \N$ free variables
$x_1, \dots,x_n$ and suppose that $\dot x_1, \dots,\dot x_n$ are $\PP$-names. 
Then for every $p\in\PP$ and every automorphism $\pi\colon \PP\to\PP$
$$p\forces \phi(\dot x_1, \dots,\dot x_n)\ \ \hbox{\rm if and only if} \ \ \pi(p)\forces 
\phi(\pi^{\dot\PP}(\dot x_1), \dots,\pi^{\dot\PP}(\dot x_n)),$$
where $\pi^{\dot\PP}$ denotes the automorphism on $\PP$-names induced by $\pi$. 
\end{lemma}

\begin{lemma}\label{permutation-symmetric} Suppose that 
$S_1, S_2\subseteq\omega_2$, $\sigma:\omega_2\rightarrow\omega_2$
is a permutation such that $\sigma=\sigma^{-1}$ and $p\in \PP$ is $(\sigma, S_1, S_2)$-symmetric.  
Suppose that $\phi(x_1, \dots,x_n)$ is a formula
 of the language of {\sf ZFC} with  $n\in \N$ free variables
$x_1, \dots,x_n$ and suppose that $\dot x_1, \dots,\dot x_n$ are nice $\PP$-names
whose supports are included in $S_1$.  Then
 $$p\forces \phi(\dot x_1, \dots,\dot x_n)\ \ \hbox{\rm if and only if} \ \  p\forces 
 \phi(\sigma^{\dot\PP}(\dot x_1), \dots,\sigma^{\dot\PP}(\dot x_n))$$
\end{lemma}
\begin{proof}
First, we will show ``$\implies$" implication. 
Note that the set $D=\{\pi(p):\pi\colon \PP\to \PP$ is a permutation satisfying 
$ \pi(q)=q $ whenever $\dom(q)\subseteq S_1\}$ is predense in 
$R=\{r\in\PP : r\upharpoonright S_1 = p\upharpoonright S_1\}$. By Lemma \ref{permutation} 
$$r\forces \phi(\dot x_1, \dots,\dot x_n)$$
for every $r\in R$. In particular 
$$p\upharpoonright S_1\forces \phi(\dot x_1, \dots,\dot x_n).$$
By Lemma \ref{permutation} applied to $\sigma^\PP$ we get that 
$$p\leq p\upharpoonright S_2 = \sigma^\PP(p\upharpoonright S_1)\forces 
\phi(\sigma^{\dot\PP}(\dot x_1), \dots,\sigma^{\dot\PP}(\dot x_n)).$$

For the reverse implication notice that the supports of 
$\sigma^{\dot\PP}(\dot x_1), \dots,\sigma^{\dot\PP}(\dot x_n)$ 
are included in $S_2$ and apply the first part of the proof to $\sigma^{-1}$. 
\end{proof}

\section{The product of $c_0(2^\omega)$ in the Cohen model}

    Every vector in $\ell_2$ is a sequences of complex numbers, 
    and every complex number is a pair of real numbers. 
    Thus, by the standard argument we may choose bijections defined in an 
    absolute way, that identify the sets
     $\ell_2 \sim \C^\N \sim \R^\N \sim  \N^\N \sim \mathcal{P}(\N)$.  
     We will use such identification in the context of names for 
     vectors - in particular, under such identification, every $v\in\ell_2$ in 
     $V^\PP$ is named by a nice $\PP$-name $\dot{v} \in V$. 

\begin{lemma}\label{permutations13}
Assume {\sf CH}.
    Suppose $A\in [\omega_2]^{\omega_2}$, $\{S_\alpha\}_{\alpha\in A}$ is a
     $\Delta$-system of distinct countable sets with root $\Delta$
      and for every $\alpha\in A$ and $l\in\N$
       we are given  a nice $\PP$-name $\dot{v}_\alpha(l)$ 
       for a vector in $\ell_2$ such that
        $\supp(\dot{v}_\alpha(l))\subseteq S_\alpha$. 
        Then there is $C\in [A]^{\omega_2}$ and 
        permutations $\sigma_{\alpha,\beta}\colon \omega_2\to\omega_2 $ such that 
    \begin{itemize}
    \item $\sigma_{\alpha, \beta}= \sigma_{\alpha, \beta}^{-1}$,
    \item $\sigma_{\alpha,\beta} [S_\alpha]= S_\beta$,
        \item $\sigma_{\alpha,\beta}\upharpoonright \Delta = Id_\Delta$,
        \item $\sigma_{\alpha,\beta}^{\dot\PP}(\dot{v}_\alpha(l))=\dot{v}_\beta(l)$
    \end{itemize}
    for all $\alpha,\beta\in C, \alpha\neq \beta$ and all $l\in\N$. 
\end{lemma}

\begin{proof}
Without loss of generality we may assume that $\Delta$ and $S_\alpha\setminus \Delta$ are infinite for all $\alpha\in A$. 
    For each $\alpha \in A$ choose a permutation $\widetilde{\sigma}_{\alpha}\colon \omega_2\rightarrow \omega_2$ 
    such that $\widetilde{\sigma}_\alpha[S_\alpha]=\N$
     and $\widetilde{\sigma}_{\alpha}\upharpoonright\Delta=\widetilde{\sigma}_{\beta}\upharpoonright\Delta$
      for all $\alpha, \beta\in A$. Then for every $\alpha\in \N$  and $l\in\N$
      the name $\widetilde{\sigma}_\alpha^{\dot\PP}(\dot{v}_\alpha(l))$
       is a nice $\PP$-name with the support included in $\N$. 
       By {\sf CH} there are only $\omega_1$ sequences of the form $(\widetilde{\sigma}_\alpha^{\dot\PP}(\dot{v}_\alpha(l)))_{l\in\N}$,
        so there is $C\in [A]^{\omega_2}$ such that 
        $\widetilde{\sigma}_{\alpha}^{\dot\PP}(\dot{v}_{\alpha}(l)) = \widetilde{\sigma}_{\beta}^{\dot\PP}(\dot{v}_{\beta}(l))$
         for $\alpha,\beta\in C$ and $l\in \N$. For $\alpha, \beta\in C$ 
         put $\widetilde{\sigma}_{\alpha,\beta} = \widetilde{\sigma}_{\beta}^{-1}\widetilde{\sigma}_\alpha$. 
         Then $\widetilde{\sigma}_{\alpha,\beta} [S_\alpha]= S_\beta$
         and $\widetilde{\sigma}_{\alpha,\beta}^{\dot\PP}(\dot{v}_\alpha(l)) =
          (\widetilde{\sigma}_\beta^{-1}\widetilde{\sigma}_\alpha)^{\dot\PP}(\dot{v}_\alpha(l)) = 
         (\widetilde{\sigma}_\beta^{-1}\widetilde{\sigma}_\beta)^{\dot\PP}(\dot{v}_\beta(l)) = \dot{v}_\beta(l)$ for $\alpha, \beta\in C, l \in \N$. Since $\widetilde{\sigma}_{\alpha}\upharpoonright\Delta=\widetilde{\sigma}_{\beta}\upharpoonright\Delta$
     we have $\widetilde{\sigma}_{\alpha,\beta}\upharpoonright \Delta = Id_\Delta$. 
         
         Now define $\sigma_{\alpha, \beta}$ by:
         \begin{itemize}
             \item $\sigma_{\alpha,\beta}\upharpoonright S_\alpha = \widetilde{\sigma}_{\alpha,\beta}\upharpoonright S_\alpha$,
             \item $\sigma_{\alpha,\beta} (\gamma) = \widetilde{\sigma}_{\alpha, \beta}^{-1}(\gamma)$ for $\gamma \in S_\beta\backslash S_\alpha$,
             \item $\sigma_{\alpha, \beta}(\gamma) = \gamma$ for $\gamma \notin S_\alpha\cup S_\beta$
         \end{itemize}
         and note that $\sigma_{\alpha, \beta}$ satisfies the required conditions. 
\end{proof}

\begin{proposition}\label{non-compact}
Assume {\sf CH}.
Suppose that
$(\dot E_{n,\alpha}:(n,\alpha)\in \N\times \omega_2)$ 
 and $(\dot B_\gamma: \gamma\in \omega_2^\N)$ 
 are such $\PP$-names that 
 for every $\alpha<\omega_2$, $n\in\N$ and $\gamma\in \omega_2^\N$ the forcing $\PP$ forces that
 \begin{itemize}
 \item $\dot E_{n,\alpha}$,  $\dot B_\gamma$ is a noncompact projection   
  in $\mathcal{B}(\ell_2)$,
  \item $\dot E_{n,\gamma(n)}\leqK \dot B_\gamma$.
  \end{itemize}
Then there are disjoint $\gamma_1,\gamma_2 \in \omega_2^\N$ such
 that 
 $$\PP\forces\dot B_{\gamma_1}\dot B_{\gamma_2}\  \hbox{\rm is non-compact.}$$ 
\end{proposition}

\begin{proof}

For $(n,\alpha)\in \N\times \omega_2$ and $l\in \N$ let $\dot{v}_{n,\alpha}(l)$
 be a nice $\PP$-name such that 
 \begin{enumerate}
    \item\label{vs-very} $\PP\forces (\dot{v}_{n,\alpha}(l))_{l\in\N}$ is a very  orthonormal sequence in $\ell_2$ ,
    \item\label{Es-vs} $\PP\forces \|\dot{E}_{n,\alpha}(\dot{v}_{n,\alpha}(l))\|\geq 1-1/(l+1)$
\end{enumerate}
The existence of such names follows from Lemma \ref{projection-very}.
For $n\in \N, \alpha\in \omega_2$ put 
\begin{gather}\label{Sns-vns}\tag{3}
S_{n,\alpha}=\bigcup_{l\in \N} \supp(\dot{v}_{n,\alpha}(l)).
\end{gather}
Each
$S_{n,\alpha}$ is countable, so by Lemma \ref{Delta_lemma_CH}, 
for each $n\in \N$ there is a set $A_n\in [\omega_2]^{\omega_2}$ 
such that $(S_{n,\alpha})_{\alpha\in A_n}$ is a $\Delta$-system with root $\Delta_n$, i.e.,
\begin{gather}\label{delta-sys}\tag{4}
S_{n,\alpha}\cap S_{n,\beta}=\Delta_n
 \end{gather}
 for all $n\in \N$ and distinct $\alpha, \beta\in A_n$.
The set
$\Delta=\bigcup_{n\in\N} \Delta_n$ is countable, 
so by  further thinning out of each $A_n$ we may assume that for every $\alpha\in A_n$ we have
$\Delta\cap S_{n,\alpha} = \Delta_n.$

By Lemma \ref{permutations13} we may also assume that for each
 $n\in \N$ and $\alpha, \beta\in A_n$
  there is a permutation $\sigma_{n,\alpha, \beta}\colon \omega_2 \to \omega_2$ such that
\begin{enumerate}\setcounter{enumi}{4}
    \item\label{sigma-inverse} $\sigma_{n,\alpha, \beta}= \sigma_{n, \alpha, \beta}^{-1}$,
    \item\label{sigma-Ss} $\sigma_{n,\alpha, \beta}[S_{n,\alpha}]=S_{n,\beta}$,
    \item\label{id-deltan} $\sigma_{n,\alpha,\beta}\upharpoonright {\Delta_n}=Id_{\Delta_n}$,
    \item\label{sigma-Es}  $\sigma_{n,\alpha,\beta}^{\dot\PP}(\dot{v}_{n,\alpha}(l))=\dot{v}_{n,\beta}(l)$.
\end{enumerate}
     
By (\ref{delta-sys}) we may use
Lemma \ref{functions} to obtain  $\gamma_\xi\in \omega_2^\N$ for every $\xi<\omega_2$ such that 
\begin{enumerate}\setcounter{enumi}{8}
    \item\label{Ss-disjoint} for distinct $(\xi,n),(\eta,m)\in\omega_2\times\N$  we have
    $$(S_{n,\gamma_\xi(n)}\backslash \Delta_n)\cap (S_{m,\gamma_\eta(m)}\backslash \Delta_m)=\varnothing,$$
    \item\label{gammas-disjoint} $\gamma_\xi\cap \gamma_\eta = \varnothing$ for $\xi, \eta<\omega_2, \xi\neq \eta$.
\end{enumerate}

For $n\in\N, \xi<\omega_2$ let $\dot{K}_{n,\xi}$ be such $\PP$-names that 
\begin{gather}\label{Ks}\tag{11}
\PP \forces  \dot{K}_{n,\xi}=
 \dot{E}_{n,\gamma_\xi(n)}-\dot{B}_{\gamma_\xi}\dot{E}_{n,\gamma_\xi(n)}.
 \end{gather}

Since by the hypothesis of the lemma, for each $n\in \N$ and $\xi<\omega_2$ we have that
 $\PP \forces \dot{E}_{n,\gamma_\xi(n)}\leqK \dot{B}_{\gamma_\xi}$, we may conclude that
$\PP \forces \dot{K}_{n,\xi} \text{ is compact}$, and so
by Lemma \ref{compact_Conway} for each $\xi<\omega_2$ there is a nice $\PP$-name $\dot{h}_\xi$ such that 
$\PP \forces \dot{h}_\xi\colon \N\rightarrow \N$  and
    for all $(n,\alpha)\in {\gamma}_\xi$ we have 
\begin{equation}\label{hs}
\begin{gathered}\tag{12}
    \PP \forces  \forall l\geq \dot{h}_\xi(\check n)  
    \ \ \ \ \ \|\dot{K}_{n,\xi}(\dot{v}_{n,\alpha}(l))\|<0.01. 
\end{gathered}
\end{equation}
Let $R_\xi=\supp(\dot{h}_\xi)$ and $S_\xi=\bigcup_{n\in\N} 
(S_{n,\gamma_\xi(n)}\backslash \Delta_n)$. 
By (\ref{Ss-disjoint}) the sequence $(S_\xi)_{\xi<\omega_2}$ 
consists of pairwise disjoint subsets of $\omega_2$. 
Since each $S_{n,\gamma_\xi(n)}$ is countable, $S_\xi$ is also countable for $\xi<\omega_2$. 

By Lemma \ref{RcapS} there are $\xi, \eta<\omega_2$ such that $\xi\neq \eta$ and 
\begin{gather}\label{RS-disjoint}\tag{13}
S_\xi\cap R_\eta = S_\eta\cap R_\xi=\varnothing. 
\end{gather}
For the rest of the proof we fix $\xi$ and $\eta$ as above.
By (\ref{gammas-disjoint}) $\gamma_\xi\cap \gamma_\eta=\varnothing$, so it is enough to show that 
\begin{gather*}
\PP \forces \dot{B}_{\gamma_\xi}\dot{B}_{\gamma_\eta} \text{ is non-compact.}
\end{gather*}
as required in the lemma. 
So suppose this is not the case and let us aim at a contradiction. Then by Lemma \ref{compact-very} there is
$p'\in \PP$ and a $\PP$-name $\dot f$ such that $p'$ forces that
\begin{itemize}
\item $\dot f\colon \N\rightarrow\N$,
\item For every very orthonormal sequence $(\dot v_l)_{l\in \N}$, 
for every $k\in \N$ and every $l\geq \dot f(k)$
$$\|\dot{B}_{\gamma_\xi}\dot{B}_{\gamma_\eta}(\dot v_l)\|\leq \frac{1}{k+1}.$$
\end{itemize}
Let $p\leq p'$ decide $m=\dot f(99)\in \N$, that is by (\ref{vs-very}), for every $n\in\N$ and every $l\geq m$ we have 
\begin{gather}\label{BB-vs}\tag{14}
p \forces \ \|\dot{B}_{\gamma_\xi}\dot{B}_{\gamma_\eta}
 (\dot{v}_{n,\gamma_\xi( n)}( l))\|\leq1/100.
\end{gather}

By (\ref{Ss-disjoint}) there exists $n\in\N$ such that 
\begin{gather}\label{dom(p)-disjoint}\tag{15}
\dom(p)\cap (S_{n,\gamma_\xi(n)} \cup S_{n,\gamma_\eta(n)})
 \subseteq \Delta_n.
 \end{gather}
 For the rest of the proof, along with  $\xi, \eta<\omega_2$ already fixed in (\ref{RS-disjoint}), we fix
 $n\in \N$ as in (\ref{dom(p)-disjoint}). To simplify the notation we put
 \begin{enumerate}\setcounter{enumi}{15}
    \item\label{sigma} $\sigma=\sigma_{n,\gamma_{\xi}(n),\gamma_{\eta}(n)}$,
    \item\label{S} $S=S_{n,\gamma_\xi(n)}$,
    \item\label{S'} $S'=S_{n,\gamma_\eta(n)}$.
\end{enumerate}
 Clearly by (\ref{RS-disjoint}) and (\ref{delta-sys}) we have 
 \begin{gather}\label{RSS'-disjoint}\tag{19}
S\cap R_\eta,  S'\cap R_\xi\subseteq \Delta_n=S\cap S'. 
\end{gather}
 
Notice that by (\ref{sigma})-(\ref{S'}), the conditions (\ref{sigma-inverse}), (\ref{sigma-Ss}), 
(\ref{delta-sys}) and (\ref{id-deltan}) yield
\begin{enumerate}\setcounter{enumi}{19}
    \item\label{sigma-inverse-new} $\sigma^{-1}=\sigma$,
    \item\label{sigma-Ss-new} $\sigma[S]=S'$, $\sigma[S']=S$, 
    \item\label{id-deltan-new} $\sigma\upharpoonright(S\cap S')=Id_{S\cap S'}$.
\end{enumerate}
So by (\ref{dom(p)-disjoint}) the condition $p$ is 
plainly $(\sigma, S, S')$-symmetric. In the next step of the proof we will use  Lemma \ref{symmetric-condition} several
times extending $p$ to conditions $s\leq p_\eta\leq p_\xi\leq p$
which remain $(\sigma, S, S')$-symmetric and decide some useful information.
  
  Since $\supp(\dot{h}_\xi)= R_\xi$, by Lemma \ref{remarkkkkkkk} there is $q_\xi\leq p$ and $k_\xi\in \N$
   such that 
$$
    \dom(q_\xi)\backslash\dom(p)\subseteq R_\xi \text{ and } q_\xi \forces \dot{h}_\xi(\check{n})=\check{k}_\xi.
$$
Put 
$$
p_\xi= q_\xi\cup \sigma(q_\xi\upharpoonright {S}).
$$

By (\ref{sigma-Ss-new}), (\ref{id-deltan-new}) and (\ref{RSS'-disjoint}) the $q_\xi$, $\sigma$, $S$ and $S'$
satisfy the hypothesis of Lemma
 \ref{symmetric-condition} and so we may conclude that  $p_\xi\in\PP$ and $p_\xi$ 
is $(\sigma, S, S')$-symmetric. By Lemma \ref{symmetric-symmetric} it is also $(\sigma, S', S)$-symmetric.

Again since $\supp(\dot{h}_\eta)= R_\eta$, by Lemma \ref{remarkkkkkkk}
  there is $q_\eta\leq p_\xi$ and $k_\eta\in\N$ such that 
$$
 \dom(q_\eta)\backslash\dom(p_\xi)\subseteq R_\eta \text{ and }   q_\eta \forces \dot{h}_\eta({n})=\check{k}_\eta.
$$
We put
$$
    p_\eta= q_\eta\cup \sigma_{n,\gamma_\eta(n),\gamma_\xi(n)}(q_\eta\upharpoonright {S_{n,\gamma_\eta(n)}})
$$
and use again  (\ref{sigma-Ss-new}), (\ref{id-deltan-new}) and (\ref{RSS'-disjoint}) 
to see that  $q_\eta$, $\sigma$, $S'$ and $S$
satisfy the hypothesis of Lemma
 \ref{symmetric-condition} and so we  conclude that  $p_\eta\in\PP$ and $p_\eta$ 
is $(\sigma, S', S)$-symmetric, which implies that it is $(\sigma, S, S')$-symmetric as well. Clearly
\begin{gather}\label{decision-ks}\tag{23}
  p_\eta\forces \dot{h}_\xi(\check{n})=\check{k}_\xi, \dot{h}_\eta({n})=\check{k}_\eta.
\end{gather}
as $p_\eta\leq q_\eta\leq q_\xi$.
For the rest of the proof fix
\begin{gather}\label{l}\tag{24}
l>\max\{k_\xi,k_\eta,m, 99 \}.
\end{gather}
Since $l>99$ by (\ref{Es-vs}) we have
\begin{gather}\label{Es-vs-l}\tag{25}
      \PP \forces \|\dot{E}_{n,\gamma_\xi(n)}(\dot{v}_{n,\gamma_\xi(n)}(l))\|\geq 0.99. 
\end{gather}

From (\ref{Sns-vns}) and (\ref{S}) 
we conclude that  $\supp(\dot{v}_{n,\gamma_\xi(n)}(l))\subseteq S$.  So by Lemma \ref{small-domain}
there is $r\leq p_\eta$ and $v\in Fin_1(\sqrt\Q)$ 
such that 
$$
\dom(r)\backslash\dom(p_\eta)\subseteq S \ \text{and}
 \ r\forces \check{v}=\dot{v}_{n,\gamma_\xi(n)}(l). 
$$
Define
$$
s=r\cup \sigma(r\upharpoonright {S}).
$$
By (\ref{sigma-inverse-new})-(\ref{id-deltan-new}) 
we see that  $r$, $\sigma$, $S'$ and $S$
satisfy the hypothesis of Lemma
 \ref{symmetric-condition} and so
$s$
is a $(\sigma,  S, S')$-symmetric element of $\PP$  satisfying
\begin{gather}\label{decision-v}\tag{26}
s\forces \check{v}=\dot{v}_{n,\gamma_\xi(n)}(l). 
\end{gather}
From (\ref{decision-v}) and (\ref{Es-vs-l}) we get that
\begin{gather}\label{xi-eps}\tag{27}
s\forces  \|\dot{E}_{n,\gamma_\xi(n)}(\check{v})\|\geq 0.99.
\end{gather}
Since $\check{v}$ is the canonical name for an element of $V$ 
we have $\sigma^{\dot\PP}(\check{v})=\check{v}$,  so
 from (\ref{Es-vs}), (\ref{sigma-Es}), (\ref{sigma}), (\ref{l}), (\ref{decision-v})
and Lemma \ref{permutation-symmetric}  we conclude that
\begin{gather}\label{eta-eps}\tag{28}
s\forces  \|\dot{E}_{n,\gamma_\eta(n)}(\check{v})\|=\|\dot{E}_{n,\gamma_\eta(n)}
(\sigma^{\dot\PP}(\dot{v}_{n,\gamma_\xi}(l)))\|=\|\dot{E}_{n,\gamma_\eta(n)}(\dot{v}_{n,\gamma_\eta}(l))\| \geq 0.99.
\end{gather}
Now from (\ref{xi-eps}), (\ref{eta-eps}) and Lemma \ref{rachunki} (2) we conclude that 
\begin{gather}\label{EEv-big}\tag{29}
    s\forces  |\langle\dot{E}_{n,\gamma_\xi(n)}(\check v), \dot{E}_{n,\gamma_\eta(n)}(\check{v})\rangle|> 1/2.
\end{gather}
 On the other hand by (\ref{BB-vs}), by the choice of $l$ in (\ref{l}) and by (\ref{decision-v}) we have that
\begin{gather}\label{BBv}\tag{30}
s\forces \|\dot{B}_{\gamma_\xi}\dot{B}_{\gamma_\eta}(\check{v})\|<1/100.
\end{gather}
Hence by (\ref{decision-v}), (\ref{Ks}), (\ref{hs}),  (\ref{eta-eps}), (\ref{BBv}), by Lemma \ref{rachunki} (1),
by the choice of $l>k_\xi, k_\eta$ in (\ref{l}) and (\ref{decision-ks}), by the Cauchy-Schwartz inequality and by the fact that
projections have norms not bigger than one and are self-adjoint  we obtain
\begin{equation}\notag
\begin{split}
    s\forces \ &|\langle \dot{E}_{n,\gamma_\xi(n)}(\check{v}), \dot{E}_{n,\gamma_\eta(n)}(\check{v})\rangle|
     =|\langle((\dot{K}_{n,\xi} +\dot{B}_{\gamma_\xi}\dot{E}_{n,\gamma_\xi(n)})(\check v),
     (\dot{K}_{n,\eta} +\dot{B}_{\gamma_\eta}\dot{E}_{n,\gamma_\eta(n)})(\check v)\rangle|\leq
     \\ &
     \leq |\langle\dot{K}_{n,\xi}(\check v), \dot{K}_{n,\eta}(\check v)\rangle|
     +  |\langle\dot{K}_{n,\xi}(\check v), \dot{B}_{\gamma_\eta}\dot{E}_{n,\gamma_\eta(n)}(\check v)\rangle|
     + |\langle\dot{B}_{\gamma_\xi}\dot{E}_{n,\gamma_\xi(n)}(\check v), \dot{K}_{n,\eta}(\check v)\rangle|
     + \\ 
    &  +|\langle\dot{B}_{\gamma_\xi}\dot{E}_{n,\gamma_\xi(n)}(\check v), 
     \dot{B}_{\gamma_\eta}\dot{E}_{n,\gamma_\eta(n)}(\check v)\rangle|
    \leq 0.03+|\langle\dot{E}_{n,\gamma_\xi(n)}(\check v), 
     \dot{B}_{\gamma_\xi}\dot{B}_{\gamma_\eta}\dot{E}_{n,\gamma_\eta(n)}(\check v)\rangle|
     \leq \\
    & \leq 0.03+\|\dot{B}_{\gamma_\xi}\dot{B}_{\gamma_\eta}(\check{v})\|
    +\|\dot{E}_{n,\gamma_\eta(n)}(\check v)-\check v\|\leq 0.04+0.15<1/2
\end{split}
\end{equation}    

which is a contradiction with (\ref{EEv-big}). 
\end{proof}

\begin{theorem}\label{theorem_embedding_calkin}
In the Cohen model  there is no embedding of $(c_0(2^\omega))^\N$ into $\mathcal{Q}(\ell_2)$. 
\end{theorem}

\begin{proof} Work in the generic extension.
Assume that $T\colon (c_0(\omega_2))^\N\rightarrow \mathcal{Q}(\ell_2)$
 is an embedding. Let $\chi_{n,\alpha}, \chi_\gamma\in 
  (c_0(\omega_2))^\N$ for $n\in\N, \alpha<\omega_2, \gamma\in \omega_2^\N$ be given by 
\begin{equation*}
\chi_{n,\alpha}(m)(\beta) =
\begin{cases}
       & 1, \ \text{if} \ (n,\alpha)=(m,\beta) \\
      &0, \ \text{otherwise}
\end{cases}
\end{equation*}
and
\begin{equation*}
\chi_{\gamma}(m)(\beta) =
\begin{cases}
       & 1, \ \text{if} \ \gamma(m)=\beta \\
      &0, \ \text{otherwise}
\end{cases}
\end{equation*}
Clearly $\chi_{n,\alpha}, \chi_\gamma$ are projections and if $\gamma(n)=\alpha$, then $\chi_{n,\alpha}\leq \chi_\gamma$.

Since embeddings preserve posets of projections, 
$T(\chi_{n,\alpha})$ and $T(\chi_\gamma)$ are projections 
in $\QQ$ for $n\in\N, \alpha\in \omega_2, \gamma \in \omega_2^\N$ and $T(\chi_{n,\alpha})\leq T(\chi_\gamma)$ 
whenever $\gamma(n)=\alpha$. By \cite[Lemma 3.1.13]{farah-book} 
there are projections $E_{n,\alpha}, B_\gamma$ in $\mathcal{B}(\ell_2)$
 such that $\pi(E_{n,\alpha})= T(\chi_{n,\alpha}), \pi(B_\gamma)=T(\chi_\gamma)$ 
 for $n\in\N, \alpha\in \omega_2, \gamma \in \omega_2^\N$
  (here $\pi\colon \mathcal{B}(\ell_2)\to \QQ= \mathcal{B}(\ell_2)/\mathcal{K}(\ell_2)$ 
  denotes the quotient map). 
  These projections satisfy the hypothesis of Proposition \ref{non-compact}, 
  so there are disjoint $\gamma_1,\gamma_2 \in \omega_2^\N$ such that
   $B_{\gamma_1}B_{\gamma_2}$ is non-compact, which contradicts the fact 
   that $T(\chi_{\gamma_1})T(\chi_{\gamma_2})=T(\chi_{\gamma_1}\chi_{\gamma_2})=T(0)=0$. 

\end{proof}

\section{Reduced products and coronas of stabilizations} 

Now we will show an application of Theorem \ref{theorem_embedding_calkin} to reduced products. 

\begin{theorem}\label{reduced_products}
    In the Cohen model, if for every $n \in\N$ a C*-algebra $\A_n$ admits a 
    pairwise orthogonal family of projections of cardinality $2^\omega$ and $\I$ is
     an ideal of subsets of $\N$ such that $\wp(\N)/\I$ is infinite, 
     then the reduced product $$\prod_{n\in\N} \A_n/\bigoplus_{\I} \A_n$$ does not embed into $\QQ$.
\end{theorem}

\begin{proof}
    By Theorem \ref{theorem_embedding_calkin} it is enough
     to construct an embedding of $(c_0(2^\omega))^\N$ into 
     $\prod_{n\in\N} \A_n/\bigoplus_{\I} \A_n$.
      Since $\wp(\N)/\I$ is infinite, there is a pairwise disjoint family $(B_m)_{m\in\N}$ 
      of subsets of $\N$ such that $B_m \notin \I$ for every $m\in\N$. 
    
    For each $n\in\N$ let $\{p_{n, \alpha}\}_{\alpha<2^\omega}$ be a family 
    of pairwise orthogonal projections in $\A_n$. Then $T_n\colon c_0(2^\omega)\to \A_n$ 
    given by $T_n ((x_\alpha)_{\alpha<2^\omega})= \sum_{\alpha<2^\omega} x_\alpha p_{k,\alpha}$
     is an embedding, and so is $T\colon (c_0(2^\omega))^\N\to \prod_{n\in\N} \A_n$
      defined by $T((s_n)_{n\in\N})=(T_n(s_{k(n)}))_{n\in \N}$ where $k(n)=m$ whenever $n\in B_m$. 
    
    We will show that $\pi\circ T\colon (c_0(2^\omega))^\N\to \prod_{n\in\N} \A_n/\bigoplus_{\I} \A_n$
     also is an embedding, where $\pi\colon  \prod_{n\in\N} \A_n \to  \prod_{n\in\N} \A_n/\bigoplus_{\I} \A_n$
      denotes the canonical quotient homomorphism. Suppose it is not. 
      Then there is a non-zero sequence $(s_n)_{n\in\N}\in\prod_{n\in \N} \A_n$, $s_n\in c_0(2^\omega)$ 
      such that $\pi\circ T((s_n)_{n\in\N})=0$ i.e. $T((s_n)_{n\in\N})\in \bigoplus_{\I} \A_n$. 
      But this is impossible, because if $s_m\neq 0$, then the sequence $T((s_n)_{n\in\N})$ is 
      constant and non-zero on indices in $B_m\notin \I$. 
\end{proof}

Taking $\I=\{\{0\}\}$ we obtain

\begin{corollary}\label{cor_products}
    In the Cohen model, if for every $n \in\N$ a C*-algebra $\mathcal{A}_n$ 
    admits a pairwise orthogonal family of projections of cardinality $2^\omega$, 
    then the product $\prod_{n\in\N} \mathcal{A}_n$ does not embed
     into $\QQ$. In particular, there is no embedding of $(\QQ)^\N$ into $\QQ$. 
\end{corollary}

Another consequence of Theorem \ref{reduced_products} is the impossibility of 
embedding the coronas of the stabilizations of C*-algebras with big families of 
pairwise orthogonal projections. For that we will need the following lemma. 

\begin{lemma}\label{blablabla_embedding}
For any C*-algebra $\A$ there is an embedding of $\A^\N/\bigoplus_{Fin} \A$ into $\mathcal{Q}(\A\otimes\mathcal{K}(\ell_2))$.    
\end{lemma}
\begin{proof}
   First, let us define an embedding $I\colon \bigoplus_{Fin}\A\to \A\otimes\mathcal{K}(\ell_2)$: 
    $$I((a_n)_{n\in\N}))= \sum_{n\in\N} a_n\otimes e_{n,n}$$
    where $a_n\in \A$ and $e_{n,m}\in \mathcal{K}(\ell_2)$ is given by 
    $$e_{n,m}((c_k)_{k\in\N})(l)= 
    \begin{cases}
        c_m,& \ \text{if} \ l=n\\
        0,& \ \text{otherwise}
    \end{cases}
    $$
    for $n,m, l\in\N$ and $(c_n)_{n\in\N}\in \ell_2$ i.e. $e_{n,m}$ is the matrix with entry $1$ at position $(n,m)$ and $0$ at other positions (so called matrix unit). 

    Now fix $A\in \A^\N, A=(a_n)_{n\in\N}$. Consider operators $L_A, R_A\colon \A\otimes \mathcal{K}(\ell_2)\to \A\otimes \mathcal{K}(\ell_2)$ defined by 
    \begin{eqnarray*}
        L_A (x\otimes e_{n,m}) = a_nx\otimes e_{n,m} \\
        R_A (x\otimes e_{n,m}) = xa_m\otimes e_{n,m}
    \end{eqnarray*}
    i.e. $L_A$ and $R_A$ are multiplications by diagonal matrices (from the left and right side respectively) with entries $(a_n)_{n\in\N}$ on the diagonal. 
    
It is easy to check that $(L_A, R_A)$ is a double centralizer of $\A\otimes \mathcal{K}(\ell_2)$ and $I_\infty\colon \A^\N\to \mathcal{M}(\A\otimes\mathcal{K}(\ell_2))$ given by $I_\infty(A)=(L_A, R_A)$ is an embedding.
Moreover, $I_\infty(A)\in \A\otimes\mathcal{K}(\ell_2)$ if and only if
$A\in \bigoplus_{Fin}\A$. Thus, 
$$J\colon \A^\N/\bigoplus_{Fin}\A \to \mathcal{Q}(\A\otimes\mathcal{K}(\ell_2))$$ 
given by $J([A])=[I_\infty(A)]$ is well-defined and is an embedding.
\end{proof}

\begin{corollary}\label{corona-no-embedding}
In the Cohen model there is no embedding of 
$\mathcal{Q}(\A\otimes \mathcal{K}(\ell_2))$ into $\QQ$
 for any C*-algebra $\A$ that admits a pairwise orthogonal
  family of projections of cardinality $2^\omega$. 
  In particular, $\mathcal{Q}(\QQ\otimes \mathcal{K}(\ell_2))$ does not embed into $\QQ$. 
\end{corollary}

\begin{proof}
    Apply Theorem \ref{reduced_products} to $\A_n=\A$ for $n\in\N, \I=Fin$ and Lemma \ref{blablabla_embedding}.
\end{proof}

\section{Concluding remarks and questions}

First let us note that the proof of \cite[Theorem 1.2]{vaccaro-ijm} implies that under the
 Open Coloring Axiom {\sf OCA} there is no embedding of $\QQ\otimes C(\omega+1)$ 
 (which is isomorphic to $C(\omega+1, \QQ)$) into $\QQ$. 
 Using techniques similar to those described in the proof of Theorem 
 \ref{reduced_products} one may embed $C(\omega+1, \QQ)$ into $\QQ^\N/\bigoplus_{Fin} \QQ$ 
 and so into $\mathcal{Q}(\QQ\otimes \mathcal{K}(\ell_2))$. In particular, like
 in our Corollary \ref{corona-no-embedding} also under {\sf OCA}
  there is no embedding of $\mathcal{Q}(\QQ\otimes \mathcal{K}(\ell_2))$ into $\QQ$.

It is  worth mentioning that by \cite[Theorem 1.2.6]{handbook-van-mill}, 
assuming {\sf CH}, the Stone-\v{C}ech remainder $\N^*$ 
of the space of natural numbers (with the discrete topology) 
is homeomorphic to the Stone-\v{C}ech remainder $X^*$ of the space
 $X=\bigsqcup_{n\in\N} \N^*$ consisting of countably many pairwise 
 disjoint copies of $\N^*$. The algebra $C(\N^*)$ is isomorphic to
  $\ell_\infty/c_0$ and the algebra $C(X^*)$ is isomorphic to 
  $(\ell_\infty/c_0)^\N/ (\bigoplus_{Fin} \ell_\infty/c_0)$. 
  Hence, under {\sf CH}, there is an isomorphism between
   the commutative C*-algebra $\ell_\infty/c_0$ and the
    reduced product $ (\ell_\infty/c_0)^\N/ (\bigoplus_{Fin} \ell_\infty/c_0)$. 
    We do not know whether such isomorphism may exist in the non-commutative versions of those algebras. 

\begin{question}
    Assume the continuum hypothesis {\sf CH}. Are the algebras 
    $\QQ$ and $\mathcal{Q}(\QQ\otimes \mathcal{K}(\ell_2))$ isomorphic?
\end{question}

Secondly, let us comment further on the relation of our main result and other 
consistent constructions of abelian C*-algebras not in $\E$ with the abelian case.
All these results are generalizations of older results
about the nonexistence of  Boolean embeddings of some Boolean algebras into
the Boolean algebra $\wp(\N)/Fin$ corresponding to $\ell_\infty/c_0$. 
The mentioned results of \cite{vignati-thesis}, \cite{vignati-mckenney} generalize results
of \cite{farah-aq}, the mentioned results of \cite{vaccaro-thesis} generalize a result from \cite{kunen-thesis}
and our results generalize a result from  \cite{dow-sat}. 
Moreover, by the results of \cite{pk-brech-univ} and \cite{pk-brech-sums} 
the Banach spaces $C([0,\omega_2])$ and $\ell_\infty(c_0(2^\omega))$ 
cannot be isomorphically embedded into the Banach space $\ell_\infty/c_0$ in the Cohen model.
So these results could immediately imply that $C_0(\omega_2)$ and $(c_0(2^\omega))^\N$ 
do not belong to $\E$ if  any of the versions of the following question had the positive answer:

\begin{question} Suppose that $\A\subseteq\QQ$ is a maximal selfadjoint abelian
subalgebra (masa) of $\QQ$. Does $\A$ embed as a Banach space into the Banach space $\ell_\infty/c_0$?
Does $\A$ embed into the algebra $L_\infty\oplus\ell_\infty/c_0$?\end{question}

Unfortunately we know very little about masas of $\QQ$. The only known isomorphic types
come from masas in $\mathcal B(\ell_2)$ and are of the form $\ell_\infty/c_0$, $L_\infty$
and their sums. Note that $L_\infty$ is isomorphic to $\ell_\infty$ as a Banach space
by a theorem of Pe\l czy\'nski (\cite{pelczynski}), so embeds as a Banach space into $\ell_\infty/c_0$
(although it may not embed as a C*-algebra consistently  by \cite{dow-fm}). 
However the algebras $C_0(\omega_2)$ or $(c_0(2^\omega))^\N$ 
embed into $L_\infty\oplus\ell_\infty/c_0$ if and only if they embed into
$\ell_\infty/c_0$ because $L_\infty$ is isomorphic to a $C(K)$, where $K$ is the Hewitt-Yosida compact space
which is c.c.c., hence the positive answer to the second version of the question is sufficient
for concluding that $C_0(\omega_2)$ and $(c_0(2^\omega))^\N$  are not in $\E$ in the
Cohen model.

Our proof of Theorem \ref{main} works the same way in the case of adding more than $\omega_2$
Cohen reals yielding the consistency of the nonembeddability of $(c_0(\omega_2))^\N$ 
into the Calkin algebra together with $2^\omega$ bigger than $\omega_2$. We provided
the proof in the case of $\omega_2$ Cohen reals because anyway the algebra $(c_0(\omega_2))^\N$
is of density continuum. 

\bibliographystyle{amsplain}

\end{document}